\documentclass[12pt,reqno]{amsart}
\usepackage{amssymb,amsthm,latexsym, amsmath, url}
\usepackage{color}
\usepackage{fullpage}
\usepackage{hyperref}

\title[Positive Solns to Singular 2nd Order BVPs]{Existence Result for Singular Second Order Dynamic Equations with Mixed Boundary Conditions}
\author{Shalmali Bandyopadhyay, Curtis Kunkel}
\address{Shalmali Bandyopadhyay, Curtis Kunkel, University of Tennessee at Martin, Martin, TN 38238, USA}

\newtheorem{theorem}{Theorem}[section]

\newtheorem{definition}[theorem]{Definition}

\begin{document}

\begin{abstract} We explore singular second-order boundary value problems with mixed boundary conditions on a general time scale. Using the lower and upper solutions method combined with the Brouwer fixed point theorem we demonstrate the existence of a positive solution and obtain the desired solution by using a sequence of solutions to a sequence of nonsingular second-order equations and passing to the limit.  
\end{abstract}
\maketitle
\section{Introduction}
Study of dynamic equations aka differential equations on time-scale is fairly new, established in 1988, as the doctoral dissertation of Stephen Hilger (founder of time-scale calculus) and has a lot of potential theoretical exploration in recent future. Basically, time-scale bridges the gap between the continuous and the discrete mathematics, which intimidated several applied mathematicians to dive into the research of time-scale.  There are several applications of dynamic equations which includes the study of population models and life cycles of Cicadas (which are only prime numbers, interestingly), epidemic models, heat transfer, study of applications of pesticides while treating mosquitoes (see \cite{thomas2009}), study of control systems in which signals are transmitted over an asynchronous network (see \cite{kozyakin2004}, \cite{siegmund2021}), neural networks and so on.\\
\par In General, Singular boundary value problems have also garnered significant attention.  Representative works include \cite{agarwal1999}, \cite{precup2016}, and \cite{rachunkova2009}.
This paper investigates the following problem
\begin{equation}
\begin{cases}
u^{\Delta\Delta}(\rho(t)) + f(t, u(t)) = 0, \quad t \in \mathbb{T}^0,\\
u^{\Delta}(0) = 0, \\
u(T) = g(T).
\end{cases}
\end{equation}
which furthers the work done previously on singular second order boundary value problems on general time scale done by Kunkel in \cite{kunkel2019} , where the author studied the following mixed boundary value problem
\begin{equation}
\begin{cases}
u^{\Delta\Delta}(\rho(t)) + f(t, u(t)) = 0, \quad t \in \mathbb{T}^0,\\
u^{\Delta}(0) = 0, \\
u(T) = 0.
\end{cases}
\end{equation}

While similar across most of the time scale, this result stands out because the time scale is completely general. It can be entirely continuous (as in \cite{kunkel2006}), entirely discrete (as in \cite{kunkel2008}), or any combination in between, as long as the underlying interval remains a closed subset of the real numbers.

So far, lower and upper solutions have been used extensively in establishing solutions of boundary value problems for finite difference equations (e.g. see\cite{bao2012}, \cite{henderson2006}, and \cite{precup2016}). The techniques presented in this paper are based primarily on lower and upper solution methods where we employ the Brouwer fixed point theorem \cite{zeidler1986} to achieve the existence of solution for regular problem (will be defined later); the problem without the singularity which will be applied to nonsingular perturbations of our nonlinear problem. This process will eventually lead to the boundary value problem by taking the limit. We will provide definitions of appropriate lower and upper solutions.

\section{Preliminaries}

In this section, we state some definitions used throughout the remainder of the paper, many of which can be found in \cite{bohner2001}.  Some definitions are required prior to the introduction of the problem we intend to solve.

\begin{definition}\label{timescale}
	A time scale is any arbitrary nonempty closed subset of the reals.
\end{definition}

For our purposes of considering a boundary value problem with boundary conditions occurring at both the lower and upper extreme values of $t$ in $\mathbb{T},$ we will specify our time scale $\mathbb{T}$ as having a minimum value of $0$ and a maximum value of $T.$  Thus, our time scale will by default be nonempty, and we further specify our time scale $\mathbb{T}$ to be any arbitrary closed subset of $[0,T].$

Bohner and Peterson standardized the notation for time scales in their text \cite{bohner2001}, and we include some of the more necessary definitions below.

\begin{definition}
	Let $\mathbb{T}$ be a time scale.  Let $t \in \mathbb{T}.$ \\
	We define the forward jump operator $\sigma : \mathbb{T} \rightarrow \mathbb{T}$ by $$\sigma(t) := \inf\{s \in \mathbb{T} : s > t\}.$$
    We define the backward jump operator $\rho : \mathbb{T} \rightarrow \mathbb{T}$ by $$\rho(t) := \sup\{s \in \mathbb{T} : s < t\}.$$
    In this definition, we specify $\inf \emptyset = \sup \mathbb{T} = T$ and $\sup \emptyset = \inf \mathbb{T} = 0.$
\end{definition}

\begin{definition}
	For the purpose of defining differentiation, we need to specify the time scale $$\mathbb{T}^k = \mathbb{T} - \{ T \}.$$
	For the purpose of defining our boundary value problem, we need to specify the time scale $$\mathbb{T}^0 = \mathbb{T} - \{ 0,T \}.$$
\end{definition}

\begin{definition}
	Assume $f: \mathbb{T} \rightarrow \mathbb{R}$ is a function and let $t \in \mathbb{T}^k.$  Then we define $f^{\Delta}(t)$ to be the number (provided it exists) with the property that given any $\varepsilon > 0,$ there is a neighborhood $U$ of $t$ such that $$\left|[f(\sigma(t)) - f(s)] - f^{\Delta}(t)[\sigma(t) - s]\right| \leq \varepsilon |\sigma(t) - s|,$$ for all $s \in U.$  We call $f^{\Delta}(t)$ the delta (or Hilger) derivative of $f$ at $t.$  We also make note that $f^{\Delta \Delta}(t) = (f^{\Delta})^{\Delta}(t).$
\end{definition}

Having introduced these definitions, we can now introduce our mixed boundary value problem with boundary conditions,  which will be our focus throughout this paper, 

\begin{equation}\label{eq}
\begin{cases}
u^{\Delta\Delta}(\rho(t)) + f(t, u(t)) = 0, \quad t \in \mathbb{T}^0,\\
u^{\Delta}(0) = 0, \\
u(T) = g(T).
\end{cases}
\end{equation}

Our goal is to prove the existence of a positive solution to this problem \eqref{eq} under certain assumptions concerning the function $f$ as explained below.  Prior to formally defining these assumptions, we first need a few more definitions.

\begin{definition}
	Define a solution to problem \eqref{eq}, to mean a function $u : \mathbb{T} \rightarrow \mathbb{R}$ such that $u$ satisfies \eqref{eq} on $\mathbb{T}^0$ and also satisfies the boundary conditions.  If $u(t) > 0$ for $t \in \mathbb{T},$ except possibly at the boundary conditions, we call $u$ a positive solution to problem \eqref{eq}.
\end{definition}

\begin{definition}
	Let $D \subseteq \mathbb{R}.$  We say $f$ is continuous on $\mathbb{T} \times D$ if $f(\cdot, x)$ is defined on $\mathbb{T}$ for each $x \in D$ and if $f(t, \cdot)$ is continuous on $D$ for each $t \in \mathbb{T}.$
\end{definition}

\begin{definition}
	Let $D \subseteq \mathbb{R},$ where $f:\mathbb{T} \times D \rightarrow \mathbb{R}.$ \\ If $D = \mathbb{R},$ then we call \eqref{eq} a regular problem. \\ If $D \subsetneq \mathbb{R}$ and $f$ has singularities on the boundary of $D,$ then we call \eqref{eq}a singular problem.
\end{definition}

We now have sufficient definitions to introduce assumptions (\textbf{A})-(\textbf{C}) that will be used throughout this paper:

\begin{description}
	\item[A] {$D = [0, \infty).$}
	\item[B] {$f$ is continuous on $\mathbb{T} \times D.$}
	\item[C] {$f(t, x)$ has a singularity at $x=0,$ i.e. $\limsup\limits_{x \rightarrow 0^{+}} |f(t, x)| = \infty,$ for $t \in \mathbb{T}.$}
\end{description}

We now can specify that our problem \eqref{eq}is a singular boundary value problem defined on our general time scale $\mathbb{T}.$  Before we can begin discussing our solution technique, however, we need to mention what we mean by integration over a general time scale $\mathbb{T}.$  

\begin{definition}
	A function $f:\mathbb{T} \rightarrow \mathbb{R}$ is called rd-continuous provided it is continuous at right-dense points in $\mathbb{T}$ and its left-sided limits exist (finite) at left-dense points in $\mathbb{T}.$  The set of rd-continuous functions $f:\mathbb{T} \rightarrow \mathbb{R}$ will be denoted by $C_{rd}.$
\end{definition}

\begin{theorem}
	Assume $f:\mathbb{T} \rightarrow \mathbb{R}.$  If $f$ is continuous, then $f$ is rd-continuous.
\end{theorem}

\begin{theorem}
	Let $a,b \in \mathbb{T}$ and $f \in C_{rd}.$
		\begin{enumerate}
		\item If $\mathbb{T} = \mathbb{R},$ then $$\int_{a}^{b} f(t) \Delta t = \int_{a}^{b} f(t) dt,$$ where the integral on the right is the usual Riemann integral from calculus.

		\item If $[a,b]$ consists only of isolated points, then
		\begin{equation*}
			\int_{a}^{b} f(t) \Delta t = \left\{
			\begin{array}{ll}
				\sum_{t\in[a,b)} \mu(t)f(t) & if \quad a < b \\
				0 & if \quad a = b \\
				- \sum_{t\in[b,a)} \mu(t)f(t) & if \quad a > b.
			\end{array}
			\right.
		\end{equation*}

		\item If $\mathbb{T} = \mathbb{Z},$ then 
		\begin{equation*}
			\int_{a}^{b} f(t) \Delta t = \left\{
			\begin{array}{ll}
				\sum_{t=a}^{b-1} f(t) & if \quad a < b \\
				0 & if \quad a = b \\
				- \sum_{t=b}^{a-1} f(t) & if \quad a > b.
			\end{array}
			\right.
		\end{equation*}
	\end{enumerate}
\end{theorem}

The previous definition and theorems are included from \cite{bohner2001} to allow the reader a better understanding of integration on a time scale.  In the next section, we relax the singular nature of \eqref{eq} and create a lower and upper solutions method for this similar regular problem.  In the subsequent section, we use this result to a perturbation of \eqref{eq} so that a sequence of regular problems are created.  Each of these having their own solution, we pass to a limit of this sequence, yielding the desired result of the singular problem, which we will show also satisfies the positivity condition.

\section{Lower and upper solutions Method}

We first consider the following regular problem to establish a sub and upper solution method. Thereafter, we implement this method to solve the singular problem. 

\begin{equation}\label{eq_reg}
\begin{cases}
u^{\Delta\Delta}(u(t)) + h(t, u(t)) = 0, \quad t \in \mathbb{T}^0\\
u^{\Delta}(0)=0\\
u(T)=g(T)
\end{cases}
\end{equation}
where $h$ is continuous on $\mathbb{T} \times \mathbb{R}$ and $g$ is Lipschitz continuous. Clearly, \eqref{eq_reg} is a regular problem and it is our current goal to establish a sub and upper solution method in order to obtain an existence result. To begin with, we first define the lower solution and upper solution to \eqref{eq_reg}.

\begin{definition}
	Let $\alpha : \mathbb{T} \rightarrow \mathbb{R}.$  We call $\alpha$ a lower solution of problem \eqref{eq_reg} if, 
	
	\begin{equation}\label{low_eq}
 \begin{cases}
  \alpha^{\Delta\Delta}(\rho(t)) + h(t, \alpha(t)) \geq 0, \quad t \in \mathbb{T}^0,\\
  \alpha^{\Delta}(0) \geq 0,\\
  \alpha(T) \leq g(T).
 \end{cases}
\end{equation}

\end{definition}
 
\begin{definition}
	Let $\beta : \mathbb{T} \rightarrow \mathbb{R}.$  We call $\beta$ a upper solution of problem \eqref{eq_reg} if, 
	
	\begin{equation}\label{up_eq}
 \begin{cases}
   \beta^{\Delta\Delta}(\rho(t)) + h(t, \beta(t)) \leq 0, \quad t \in \mathbb{T}^0,\\
   \beta^{\Delta}(0) \leq 0,\\
   \beta(T) \geq g(T)
 \end{cases}
\end{equation}
\end{definition}

\begin{theorem}\label{thm_reg}
	Let $\alpha$ and $\beta$ be lower and upper solutions of the regular problem \eqref{eq_reg} respectively, such that, $\alpha \leq \beta$ on $\mathbb{T}.$  Let $h(t, x, y)$ be continuous on $\mathbb{T} \times \mathbb{R}^2$ and non-increasing in its $y$-variable.  Then \eqref{eq_reg} has a solution $u$ satisfying $$\alpha(t) \leq u(t) \leq \beta(t), \quad t \in \mathbb{T}.$$
\end{theorem}
\begin{proof}
This proof involves several steps as follows, 

    \textbf{Step 1:} For $t \in \mathbb{T}$ and $x \in \mathbb{R},$ let us define  
    
	\begin{equation}\label{tilde}
	\tilde{h}(t, x) = \left\{\begin{array}{ll}
	h(t, \beta(t)) - \frac{x - \beta(t)}{x - \beta(t) + 1}, & x > \beta(t), \\
	h(t, x), & \alpha(t) \leq x \leq \beta(t), \\
	h(t, \alpha(t)) + \frac{\alpha(t) - x}{\alpha(t) - x + 1}, & x < \alpha(t),
	\end{array}\right.
	\end{equation}
$\tilde{h}$ is continuous on $\mathbb{T} \times \mathbb{R}$ by construction and hence bounded. Therefore, there exists $M>0$ so that $$\left|\tilde{h}(t, x)\right| \leq M,$$ for all $t \in \mathbb{T}$ and $x \in \mathbb{R}.$
We now prove existence of solutions to the following auxiliary problem:	
\begin{equation}\label{aux}
\begin{cases}
u^{\Delta\Delta}(\rho(t)) + \tilde{h}(t, u(t)) = 0, \quad t \in \mathbb{T}^0\\
u^{\Delta}(0)=0\\
u(T)=g(T).
\end{cases}
\end{equation}
	
	\textbf{Step 2:} For this existence result, we lay the foundation to use the Brouwer fixed point theorem.  To this end, define $$E = \{u : \mathbb{T} \rightarrow \mathbb{R} | u^{\Delta}(0) = 0;\quad u(T) = g(T)\}.$$
	Also, define $$||u|| = \sup\left\{|u(t)| : t \in \mathbb{T}\right\}.$$	
	
	Given $E$ and $||\cdot||,$ we say $E$ is a Banach space.  Further, we define an operator $\mathcal{T} : E \rightarrow E$ by
	
	\begin{equation}\label{T_op}
	\left(\mathcal{T}u\right)(t) = g(T)+\int_{t}^{T} \int_{0}^{r} \tilde{h}(s, u(s)) \Delta s \Delta r.
	\end{equation}
	
By construction, $\mathcal{T}$ is a continuous operator.  Moreover, from the bounds placed on $\tilde{h}$ in \textbf{Step 1} and from \eqref{T_op}, if $r>MT^2,$ then $\mathcal{T}\left(\overline{B(r)}\right) \subseteq \overline{B(r)},$ where $B(r) := \left\{u \in E : ||u|| < r\right\}.$  Hence, by the Brouwer fixed point theorem \cite{zeidler1986}, there exists $u \in \overline{B(r)}$ such that $u = \mathcal{T}u.$
	
    \textbf{Step 3:} The next step is to prove that if $u$ is a fixed point of $\mathcal{T}$ then $u$ is a solution to the problem \eqref{aux} and vice versa.
    \par let us first assume that $u$ solves the problem \eqref{aux}. This immediately implies $u \in E$ and we get,
    \begin{eqnarray*}
    	\int_{0}^{t} u^{\Delta\Delta}(s) \Delta s & = & u^{\Delta}(t) - u^{\Delta}(0) \\
    	& = & u^{\Delta}(t),
    \end{eqnarray*}
    and $$u^{\Delta}(t) = - \int_{0}^{t} \tilde{h}(s, u(s)) \Delta s. $$
    We also have,
    \begin{eqnarray*}
    	\int_{t}^{T} u^{\Delta}(r) \Delta r & = & u(T) - u(t) \\
    	& = & g(T)- u(t).
    \end{eqnarray*}
 Therefore, combining the above two equations, we can conclude that 
    \begin{eqnarray*}
    	u(t) & = & g(T)- \left(- \int_{0}^{r} \tilde{h}(s, u(s)) \Delta s \Delta r\right) \\
    	& = & g(T) + \int_{t}^{T} \int_{0}^{r} \tilde{h}(s, u(s)) \Delta s \Delta r.
    \end{eqnarray*}

	Thus, $u = \mathcal{T}(u).$
	
\par Next we assume that $u$ is a fixed point of $\mathcal{T},$ i.e. $u = \mathcal{T}u.$ Then $u \in E$ and 
\begin{equation*}
    u=T(u)=g(T)+\displaystyle\int_{t}^{T}\int_{0}^{t}\tilde{h}(s,u(s))\Delta s\Delta r
\end{equation*}
Hence, we have, $u(T)=g(T)$ and $u^{\Delta}(t) =  - \int_{0}^{t} \tilde{h}(s, u(s)) \Delta s,$ which implies $u^{\Delta}(0) = 0$ and  $u^{\Delta\Delta}(\rho(t)) = - \tilde{h}(t, u(t)).$  Thus, $u$ is a solution to \eqref{aux}.
		
\textbf{Step 4:} The final step is to show that solutions of \eqref{aux} satisfy $$\alpha(t) \leq u(t) \leq \beta(t), \quad t \in \mathbb{T}.$$
	
	To this end, without loss of generality consider the case of obtaining $u(t) \leq \beta(t),$ and let $v(t) = u(t) - \beta(t).$  For the purpose of establishing a contradiction, assume that $\max\{v(t) | t \in \mathbb{T}\} := v(l) > 0.$
	From \eqref{aux} and \eqref{up_eq}'s boundary conditions,
 \begin{eqnarray*}
   v(T) & = & u(T)-\beta(T)  \\
        & = & g (T) - (g(T)+\epsilon_1), \mbox{ for some } \epsilon_1 >0\\
   & = & \epsilon_1 < 0
 \end{eqnarray*}
 Therefore, by definition of $l$, $l \neq T$. Similarly, 
 \begin{eqnarray*}
     v^{\Delta}(0) & = & u^{\Delta}(0)-\beta^{\Delta}(0)  \\
        & = & -\epsilon_2, \mbox{ for some } \epsilon_2 >0\\
   & = & \epsilon_2 > 0
 \end{eqnarray*}
 Therefore, since the slope at $v(0)$ is positive, $v(0)$ can not be the maximum. Thus, $l \neq 0$. Hence, $l$ must be an interior point in $\mathbb{T},$ i.e. $l \in \mathbb{T}^0.$  Thus, $v^{\Delta}(\rho(l)) \geq 0$ and $v^{\Delta}(l) \leq 0,$ forcing $v^{\Delta \Delta}(\rho(l)) \leq 0.$  
		Therefore,
	\begin{equation}\label{1-4-ineq}
	u^{\Delta\Delta}(\rho(l)) - \beta^{\Delta\Delta}(\rho(l)) \leq 0.
	\end{equation}
	
	On the other hand, we have from \eqref{aux} and \eqref{tilde} that
   	\begin{eqnarray*}
		v^{\Delta \Delta}(\rho(l)) & = & u^{\Delta\Delta}(\rho(l)) - \beta^{\Delta\Delta}(\rho(l)) \\
		& = & -\tilde{h}(l, u(l)) - \beta^{\Delta\Delta}(\rho(l)) \\
		& = & -h(l, \beta(l)) + \frac{u(l) - \beta(l)}{u(l) - \beta(l) + 1} - \beta^{\Delta\Delta}(\rho(l)) \\
    	& \geq & \frac{v(l)}{v(l)+1} \\
    	& > & 0.
	\end{eqnarray*} 
	Hence a contradiction to \eqref{1-4-ineq} and we conclude that $\max\{v(t) | t \in \mathbb{T}\} \leq 0.$  Hence, $v(t) \leq 0$ for all $t \in \mathbb{T},$ or rather $$u(t) \leq \beta(t), \quad \text{for all } t \in \mathbb{T}.$$
	
	A similar argument shows that $\alpha(t) \leq u(t)$ for all $t \in \mathbb{T}.$
	
	Thus, our conclusion holds and the proof is complete.	
\end{proof}

\section{Main Result}
In this section, we make use of Theorem \ref{thm_reg} to obtain positive solutions to the singular problem \eqref{eq}.  In particular, in applying Theorem \ref{thm_reg}, we deal with a sequence of regular perturbations of \eqref{eq}.  Ultimately, we obtain a desired solution by passing to the limit on a sequence of solutions for the perturbations.

\begin{theorem}\label{main}
	Assume conditions \emph{(\textbf{A})}, \emph{(\textbf{B})}, and \emph{(\textbf{C})} hold, along with the following:
	\begin{description}
		\item[D] There exists $c \in (0,\infty)$ so that $f(t,c) \leq 0,$ for all $t \in \mathbb{T}^0.$
		\item[E] There exists $\delta > 0$ so that $f(t,x) > 0$ for all $t \in (T-\delta,T) \cap \mathbb{T}$ and $x \in \left(0, \frac{c}{2}\right).$
		\item[F] $\lim\limits_{x \rightarrow 0^{+}} f(t,x) = \infty$ for $t \in \mathbb{T}.$ 
        \item[G] $0< g(T) <u(0)$ 
	\end{description}
    Then, \eqref{eq} has a solution $u$ satisfying $$0 < u(t) \leq c, \quad t \in \mathbb{T}^k.$$
\end{theorem}
\begin{proof}
	We begin by modifying our given function $f$ as follows.
	
	For $k > 0, t \in \mathbb{T},$ define
	\begin{equation*}
	f_{k}(t,x) = \left\{\begin{array}{ll}
	f\left(t,|x|\right), & \text{if } |x| \geq \frac{1}{k} \\
	f\left(t,\frac{1}{k}\right), & \text{if } |x| < \frac{1}{k}.
	\end{array}\right.
	\end{equation*}
	Then, $f_{k}$ is continuous on $\mathbb{T} \times \mathbb{R}.$
	
	Assumption (\textbf{F}) implies that there exists a $k_{0}$ such that, for all $k \geq k_{0},$ $$f_{k}(t,0) = f\left(t,\frac{1}{k}\right) > 0, \quad \text{for all } t \in \mathbb{T}.$$
	
	We now consider the boundary value problem
	\begin{equation}\label{eq_k}
	u^{\Delta\Delta}(\rho(t)) + f_{k}(t, u(t)) = 0, \quad t \in \mathbb{T}^0,
	\end{equation}
	satisfying boundary conditions \eqref{eq}.
	
	Now, let $\alpha(t) = 0$ and $\beta(t) = c.$  Then, for each $k \geq k_{0},$ $\alpha$ and $\beta$ are lower and upper solutions of \eqref{eq_k} respectively.  Also, $\alpha(t) \leq \beta(t)$ for $t \in \mathbb{T}.$  Thus, by Theorem \ref{thm_reg}, for each $k \geq k_{0},$ there exists a solution $u_k$ to each problem \eqref{eq_k} that satisfies $0 \leq u_{k}(t) \leq c,$ for $t \in \mathbb{T}.$
	
	Let $\mathbb{T}_1 = \mathbb{T} \cap (0, T-\delta)$ and $\mathbb{T}_2 = \mathbb{T} \cap (T-\delta, T).$
	
	Now, for each $k \geq k_0,$ there exists $\varepsilon \in \left(0, \frac{c}{2}\right)$ such that if $k_{\varepsilon} \geq k_{0},$ then 
	\begin{equation}\label{fk}
		f_{k}(t, x) > c, \quad t \in \mathbb{T}, x \in (0, \varepsilon].
	\end{equation}
	
	For the sake of establishing a contradiction, assume that for $k \geq k_{\varepsilon} \geq k_0,$ we have that $u_{k_{\varepsilon}}(t) < \varepsilon_1,$ where 
	\begin{equation*}
		\varepsilon_1 = \left\{
		\begin{array}{ll}
			\varepsilon, & t \in \mathbb{T}_1, \\
			\frac{\varepsilon}{\delta}(T-t), & t \in \mathbb{T}_2.
		\end{array}
		\right.
	\end{equation*}
	
	Now,
	\begin{eqnarray*}
		-u_{k_{\varepsilon}}(t) & = &-g(T)- \int_{t}^{T} u_{k_{\varepsilon}}^{\Delta}(r) \Delta r \\
		& = & -g(T)- \int_{t}^{T} \int_{0}^{r} f_{k_{\varepsilon}}(s, u_{k_{\varepsilon}}(s)) \Delta s \Delta r \\
		& = & -g(T)- \int_{t}^{T-\delta} \int_{0}^{r} f_{k_{\varepsilon}}(s, u_{k_{\varepsilon}}(s)) \Delta s \Delta r \\
		&   & - \int_{T-\delta}^{T} \int_{0}^{r} f_{k_{\varepsilon}}(s, u_{k_{\varepsilon}}(s)) \Delta s \Delta r \\
		& < & - \int_{T-\delta}^{T} \int_{0}^{r} c \Delta s \Delta r \\
		& = & - \frac{c}{2} \left((T-\delta)^2 - t^2\right).
	\end{eqnarray*}
	
	First, consider $t \in \mathbb{T}_1.$  We have that $$u_{k_{\varepsilon}}(t) > \frac{c}{2} \left((T-\delta)^2 - t^2\right).$$  However, this implies that $u_{k_{\varepsilon}}^{\Delta\Delta}(\rho(t)) > -c,$ which leads to $f_{k_{\varepsilon}}(t, u_{k_{\varepsilon}}(t)) < c,$ a contradiction to \eqref{fk}.  Hence, for $t \in \mathbb{T}_1,$ we have $$u_{k_{\varepsilon}}(t) \geq \varepsilon_1 = \varepsilon.$$
	
	Now, consider $t \in \mathbb{T}_2.$  We have that $u_{k_{\varepsilon}} > \frac{c}{2} \left((T-\delta)^2 - t^2\right).$  This again implies that $u_{k_{\varepsilon}}^{\Delta\Delta}(\rho(t)) > -c.$  We also have, from assumption (\textbf{E}), that $$u_{k_{\varepsilon}}^{\Delta\Delta}(\rho(t)) = - f_{k_{\varepsilon}}(t, u_{k_{\varepsilon}}(t)) < 0.$$  Hence, $$-c < u_{k_{\varepsilon}}^{\Delta\Delta}(\rho(t)) < 0,$$ making $u_{k_{\varepsilon}}$ concave in this interval.  We also know, by continuity, that $u_{k_{\varepsilon}}(T-\delta) \geq \varepsilon$ and $u_{k_{\varepsilon}}(T) = 0.$  Therefore, by continuity, for $t \in \mathbb{T}_2,$ $$u_{k_{\varepsilon}}(t) > \varepsilon_1 = \frac{\varepsilon}{\delta} (T-t).$$
	
	Hence, $0 < \varepsilon_1 < u_{k_{\varepsilon}}(t)$ for $t \in \mathbb{T}^0.$
	Thus, for $k \geq k_{\varepsilon},$ we can choose a subsequence $\left\{u_{k_{n}}(t)\right\} \subseteq \left\{u_{k}(t)\right\}$ so that $$\lim\limits_{n \rightarrow \infty} u_{k_{n}}(t) = u(t), \quad t \in \mathbb{T},$$ and note that $u(t) \in E,$ where $E$ is defined as in the proof of Theorem \ref{thm_reg}.
     
    Moreover, for sufficiently large $n,$ $$u_{k_{n}} = g(T)+\int_{t}^{T} \int_{0}^{r} f(s, u_{k_{n}}(s)) \Delta s \Delta r.$$  And from the continuity of $f,$ as we let $n \rightarrow \infty,$ we get $$u(t) =g(T)\int_{t}^{T} \int_{0}^{r} f(s, u(s)) \Delta s \Delta r.$$  Hence, $$u^{\Delta\Delta}(\rho(t)) = -f(t, u(t)),$$ with the desired inequality satisfied, specifically, $0 < u(t) < c$ on $\mathbb{T}^k.$

\end{proof}

\section{Example}
Let $\mathbb{T}$ be as given following Definition \ref{timescale}. Let $\alpha \in [0, \infty)$, $c, \beta \in (0, \infty),$ and $a : \mathbb{T} \rightarrow \mathbb{R}.$  Then, by Theorem \ref{main}, the problem
\begin{tiny}
	\begin{equation*}
	u^{\Delta\Delta}(\rho(t)) + \left(a(t) + \left(u(t)\right)^\alpha + \left(u(t)\right)^{-\beta}\right)\left(c - u(t)\right) - \left(u^{\Delta}(\rho(t))\right)^3 = 0, \quad t \in \mathbb{T}^k,
	\end{equation*}
\end{tiny}
along with the boundary conditions \eqref{eq}, has a solution $u$ satisfying the desired inequality.  It is worth noting that although the function $f$ in problem \eqref{eq} does not depend on $u^{\Delta}(\rho(t))$ explicitly, $u^{\Delta}(\rho(t))$ is well defined on $\mathbb{T}^k$ and in many cases can be rewritten simply in terms of $u(t).$



\end{document}